\documentclass[a4paper,10pt]{amsart}
\usepackage{pb-diagram}
\usepackage{xypic}
\usepackage{amscd}
\usepackage [curve]{xy}
\usepackage{bm}
\usepackage{amsmath,amssymb,amscd,bbm,amsthm,mathrsfs,hyperref}
\usepackage{latexsym}
\usepackage{amsfonts}
\usepackage{amsthm}
\usepackage{indentfirst}
 \newtheorem{thm}{Theorem}[section]
 \newtheorem{cor}[thm]{Corollary}
 \newtheorem{lem}[thm]{Lemma}
 \newtheorem{prop}[thm]{Proposition}
 \newtheorem{defn}[thm]{Definition}
 
 \newtheorem{rem}[thm]{Remark}

\def\k{\mathbbm{k}}

 \newcommand{\Hom}{\mathrm{Hom}}

\title{Quillen-Suslin Theorem for connected cochain DG algebras}
\author{X.-F. Mao}
\address{Department of Mathematics, Shanghai University, Shanghai, China, 200444}
\address{Newtouch center for Mathematics of Shanghai University, Shanghai, China, 200444}
\email{xuefengmao@shu.edu.cn}
\author{B.-Y. Zhu}
\address{Department of Mathematics, Shanghai University, Shanghai 200444, China}
\email{zhubiyan@shu.edu.cn}

\date{}

\begin{document}
 \def\abstactname{abstract}
\begin{abstract}
Let $\mathscr{A}$ be a connected cochain DG algebra and $P$ a DG $\mathscr{A}$-module such that its underlying graded module $P^{\#}$ is a finitely generated $\mathscr{A}^{\#}$-module. We show that $P$ is semi-free if it is semi-projective and it is categorically free if it is categorically projective. It can be considered as a generalization of the well-known Quillen-Suslin Theorem in DG context. As an application, we show that the ghost length and the cone length of a compact DG module coincide.

\end{abstract}



\maketitle
\section*{introduction}

 All along this paper, $\k$ be a field and $n\in \mathbb{Z}^{+}$. In the mid-1950s, Serre proved that every finitely generated module over the polynomial algebra $\k[x_1,\cdots, x_n]$ has a finitely generated free complement.  And he posed the problem: are projective modules over $\k[x_1,\cdots, x_n]$ free? In the case $n=1$, it is obviously true since $\k[x]$ is a principal ideal domain. In 1958, Seshadri proved that
 finitely generated projective $R[x]$-modules are free when $R$ is a principal ideal domain. In particular, setting $R=\k[x_1]$ gives an affirmative answer to Serre's problem when $n=2$. There was much interest in this problem for $n\ge 3$; indeed, it is an important impetus in the development of algebraic K-theory. Remarkably, Quillen \cite{Qui} and Suslin \cite{Sus} independently showed in 1976 that every finitely
generated, projective $\k[x_1,\cdots, x_n]$-module is free, a result that is now known as the Quillen-Suslin Theorem. A number of algorithms for this theorem have been presented in \cite{FG},\cite{LS1} and \cite{Fit}. The Quillen-Suslin Theorem has also found applications in signal processing (cf. \cite{YP}) and control theory (cf. \cite{BBV}).
In 1997, Laubenbacher and Woodburn \cite{LW} presented an algorithm for the Quillen-Suslin Theorem for coordinate rings of affine toric varieties.
For quotients of polynomial rings by monomial ideals, Vorst \cite{Vor} proved that finitely generated projective modules over such algebras are free. In 2000,
 Laubenbacher and Schlauch \cite{LS2} presented an algorithm quotients of the form $\k[x_0,\cdots,x_n]/I$ with $I$ a monomial ideal.

In this paper, we want to generalize the Quillen-Suslin Theorem to DG homological algebra. On a conceptual level, the main novelty is the study of the internal structure of a category of DG modules from a point of view inspired by classical homological algebra. This lead us to identify various versions of freeness and projectivity, and to investigate the relations between. When it comes to free (or projective) objects in the category of DG modules over a DG algebra, we see that they are always homotopically trivial, so they carry little information of the category of DG modulmes. In particular, the standard approach to homological algebra of module over a ring, via free resolution (or projective resolution), leads to a dead end in the cases of DG module over a DG algebra. Semi-free DG modules and semi-projective DG modules serve as a remedy for this difficulty. The semi-free (semi-projective resp.) DG modules retain some characteristics of freeness (projectivity resp.).
 The diagrams below, quoted from \cite{AFH}, respectively relates the various concepts of freeness and projectivity for DG modules:
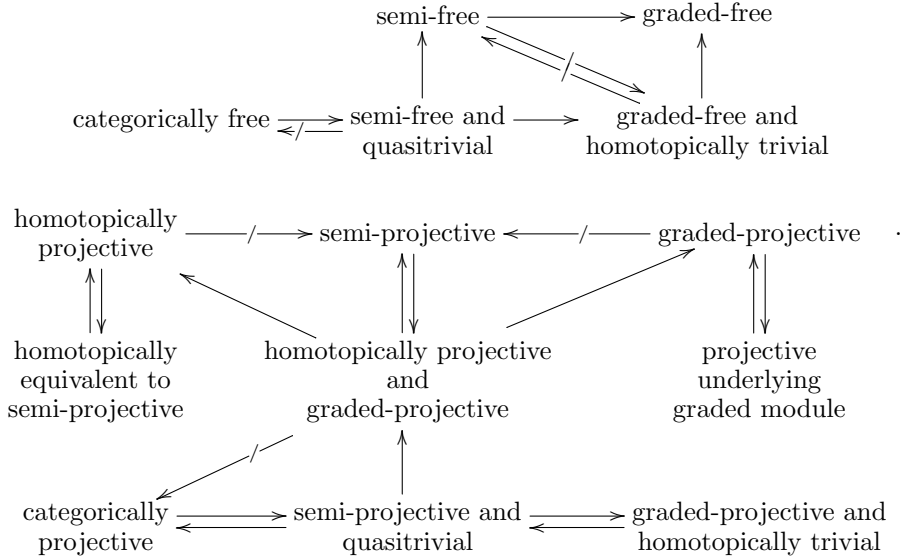
\begin{align*}
    \xymatrix{
    &\txt{semi-free}\ar@<-.8ex>[r]\ar@<.5ex>[dr]|{/} & \txt{graded-free} \\
    \txt{categorically free\\~}\ar@<.5ex>[r] & \txt{semi-free and\\ quasitrivial} \ar@<.5ex>[l]|{/} \ar@<.5ex>[r] \ar@<.5ex>[u] & \txt{graded-free and \\ homotopically trivial} \ar@<.5ex>[ul]|{/} \ar@<.5ex>[u]}
\end{align*}

\begin{align*}
    \xymatrix{
    \txt{homotopically\\projective}\ar@<-.5ex>[r]|{/}\ar@<.5ex>[d] & \txt{semi-projective}\ar@<.5ex>[d]& \txt{graded-projective}\ar@<.5ex>[l]|{/} \ar@<.5ex>[d]\\
    \txt{homotopically\\equivalent to\\semi-projective}\ar@<.5ex>[u] & \txt{homotopically projective\\and\\graded-projective} \ar@<.5ex>[lu]\ar@<.5ex>[ru]\ar@<.5ex>[u]\ar@<.5ex>[dl]|{/} & \txt{projective\\underlying\\graded module} \ar@<.5ex>[u]\\
    \txt{categorically\\projective}\ar@<.5ex>[r] & \txt{semi-projective and\\ quasitrivial} \ar@<.5ex>[l] \ar@<.5ex>[r] \ar@<.5ex>[u] & \txt{graded-projective and \\ homotopically trivial} \ar@<.5ex>[l]}.
\end{align*}
 One well known fact is that a DG module is semi-projective (resp. categorically projective) if and only if it is a direct summand of a semi-free (resp. categorically free) DG module. Motivated from Quillen-Suslin Theorem, it is natural for one to ask whether a finitely generated semi-projective (resp. categorically projective) left DG module over a connected cochain DG algebra is semi-free (resp. categorically free)?
We give an affirmative answer to the above question. As applications, we find that our results are useful in the study of homological invariants for DG modules.

In group theory, the terminology `class'
 is used to measure the shortest length of a
filtration with sub-quotients of certain type.
Avramov-Buchweitz-Iyengar  \cite{ABI} introduced free class,
projective class and flat class for differential modules over a
commutative ring.  Inspired from their work, the first author and Wu \cite{MW3} introduced the notion of `DG free class' for semi-free DG modules over a DG algebra $\mathscr{A}$.
In brief, the DG free class of a semi-free DG $\mathscr{A}$-module $F$ is the shortest length of all strictly
increasing semi-free filtrations. For a more general DG module $M$, the first author and Wu \cite{MW3} introduced the invariant `cone length', which is the minimal DG free class of all its semi-free resolutions. The studies of these invariants for DG modules can be traced back to Carlsson's work in 1980s.
In \cite{Car}, Carlsson studied `free class' of solvable free DG
modules over a graded polynomial ring $R$ in $n$ variables of positive degree. By \cite[Theorem 16]{Car}, one sees that the cone length of any
totally finite DG $R$-module $M$ satisfies the inequality $\mathrm{cl}_RM\le n$. Note that the invariant `$\mathrm{cl}_RM$' was denoted by `$l(M)$' in \cite{Car}. The invariant `cone length' of a DG module plays a similar role in DG
homological algebra as the `projective dimension' of a module over a ring
does in classic homological ring theory (cf.\cite{MW3}).

Christensen \cite{Chr} introduces the invariant `ghost length' for objects in triangulated categories. This concept can be applied to complexes and DG modules.  The ghost length together with other invariants, including level, cone length and trivial
category, of DG modules are studied in \cite{Kur1,Kur2, Mao1}.
In \cite{HL1,HL2}, Hovey-Lockridge introduce the ghost dimension of a ring, which is defined as the maximum ghost lengthes of all perfect complexes.
One can naturally extend this definition to DG algebras.
Another important invariant of DG modules is called level. In the derived category $\mathscr{D}(\mathscr{A})$ of a DG algebra $\mathscr{A}$,
the level of a DG $\mathscr{A}$-module $M$ counts
the number of steps required to build $M$ out of $\mathscr{A}$ via triangles.  In \cite{ABIM}, some important and fundamental
properties of the level of DG modules are investigated.  To study topological spaces with
categorical representation theory, Kuriayashi \cite{Kur1} introduced the
levels for space and gave a
general method for computing the level of a space and studied the
relationship between the level and other topological invariants such
as Lusternik-Schnirelmann category. Later,
Schmidt \cite{Sch} used the properties of level of DG modules to study the
structure of the Auslander-Reiten quiver of some important cochain
DG algebra.

In
\cite{Mao1}, the first author proved that the cone length of a DG module is an upper bound of the ghost length when the base DG algebra is local chain. One can check easily that it is also true for the case of connected cochain DG algebras. For any compact DG $\mathscr{A}$-module, we show that  its cone length coincides with its ghost length.
Using the
well known ghost lemma for triangulated categories, Kuribayashi \cite{Kur1}
proved
 $$\mathrm{gh.len}_{\mathscr{A}}M+1\le
  \mathrm{level}_{\mathscr{D}(\mathscr{A})}^{\mathscr{A}}(M),$$
  for any $M\in \mathscr{D}(\mathscr{A})$. We show that  $$\mathrm{gh.len}_{\mathscr{A}}M+1 =
  \mathrm{level}_{\mathscr{D}(\mathscr{A})}^{\mathscr{A}}(M),$$
  when $M\in \mathscr{D}^c(\mathscr{A})$ and $\mathscr{A}$ is a connected cochain DG algebra.

\section{Notations and conventions}
In this section, we introduce some notation and conventions on DG algebras and DG modules.
There is some
overlap here with the papers \cite{MW1, MW2,FJ}. It is assumed that
the reader is familiar with basics on DG modules, triangulated
categories and derived categories. If this is not the case, we refer
to  \cite{Nee, Wei} for more details on them.

Throughout this paper, $\k$ is an algebraically closed field of characteristic $0$.
For any $\k$-vector space $V$, we write $V^*=\Hom_{\k}(V,\k)$. Assume that $\{e_i|i\in I\}$ is a basis of a finite-dimensional $\k$-vector space $V$.  We denote the dual basis of $V$ by $\{e_i^*|i\in I\}$, i.e., $\{e_i^*|i\in I\}$ is a basis of $V^*$ such that $e_i^*(e_j)=\delta_{i,j}$.

\subsection{Connected cochain DG algebras}
Let $\mathscr{A}$ be a $\mathbb{Z}$-graded $\k$-algebra. If there is a $\k$-linear map $\partial_{\mathscr{A}} : \mathscr{A} \to \mathscr{A}$ of degree 1 such that $\partial_{\mathscr{A}}^{2}=0$ and \[
    \partial_{\mathscr{A} }\left( ab \right)=\partial_{\mathscr{A} } \left( a \right)b+\left( -1 \right)^{ | a | }a\partial_{\mathscr{A} }\left( b \right)
\]
for all graded elements $a,b \in \mathscr{A}$, then $\mathscr{A}$ is called a cochain DG algebra. A cochain DG algebra $\mathscr{A}$ is called non-trivial if $\partial_{\mathscr{A}} \ne 0$.
For any cochain DG algebra $\mathscr{A}$, its underlying graded algebra is denoted by $\mathscr{A}^{\#}$. A cochain DG algebra $\mathscr{A}$ is called connected if
$\mathscr{A}^{\#}$ is a connected graded algebra. Throughout the paper, $\mathscr{A}$ will always be a connected cochain DG algebra, if no special assumption is emphasized.

The cohomology graded algebra of $\mathscr{A}$ is the graded algebra $$H(\mathscr{A})=\bigoplus_{i\in \Bbb{Z}}\frac{\mathrm{ker}(\partial_{\mathscr{A}}^i)}{\mathrm{im}(\partial_{\mathscr{A}}^{i-1})}.$$
 For any cocycle element $z\in \mathrm{ker}(\partial_{\mathscr{A}}^i)$, we write $\lceil z \rceil$ as the cohomology class in $H(\mathscr{A})$ represented by $z$.  A cochain algebra $\mathscr{A}$ is called connected if its underlying graded algebra $\mathscr{A}^{\#}$ is a connected graded algebra. One sees that $H(\mathscr{A})$ is a connected graded algebra if $\mathscr{A}$ is a connected cochain DG algebra.
 For any connected DG algebra $\mathscr{A}$,
   we write $\frak{m}$ as the maximal DG ideal $\mathscr{A}^{>0}$ of $\mathscr{A}$.
 We denote by $\mathscr{A}\!^{op}$ its opposite DG
algebra, whose multiplication is defined as
 $a \cdot b = (-1)^{|a|\cdot|b|}ba$ for all
graded elements $a$ and $b$ in $\mathscr{A}$.
Via the canonical surjection $\varepsilon: \mathscr{A}\to \k$, $\k$ is both a DG
$\mathscr{A}$-module and a DG $\mathscr{A}\!^{op}$-module. It is easy to check that the enveloping DG algebra $\mathscr{A}^e = \mathscr{A}\otimes \mathscr{A}\!^{op}$ of $\mathscr{A}$
is also a connected cochain DG algebra with $H(\mathscr{A}^e)\cong H(\mathscr{A})^e$, $\partial_{\mathscr{A}^e}=\partial_{\mathscr{A}}\otimes \mathscr{A}^{op} + \mathscr{A}\otimes
\partial_{\mathscr{A}^{op}}$ and  $\frak{m}_{\mathscr{A}^e}=\frak{m}_{\mathscr{A}}\otimes \mathscr{A}^{op}+\mathscr{A}\otimes \frak{m}_{\mathscr{A}^{op}}$.
\subsection{DG modules}
 A left DG module over $\mathscr{A}$ (DG $\mathscr{A}$-module for
short) is a complex $(M,\partial_{M})$ together with a left
multiplication $\mathscr{A}\otimes M \to M$ such that $M$ is a left graded
module over $\mathscr{A}$ and the differential $\partial_{M}$ of $M$ satisfies
the Leibniz rule
\begin{align*}
\partial_{M}(am)=\partial_{\mathscr{A}}(a)m + (-1)^{|a|}a\partial_{M}(m)
\end{align*}
for all graded elements $a \in \mathscr{A}, \, m \in M$. A right DG module
over $\mathscr{A}$ is defined similarly. One sees easily that right DG modules over $\mathscr{A}$ can be
identified with DG $\mathscr{A}\!^{op}$-modules.

For any $i\in \Bbb{Z}$, the $i$-th suspension of a DG
$\mathscr{A}$-module $M$ is the DG $\mathscr{A}$-module $\Sigma^iM$ defined by $(\Sigma^iM)^j=
M^{j+i}$. If $m \in M^l,$ the corresponding element in $(\Sigma^i
M)^{l-i}$ is denoted by $\Sigma^i m$. The action of $\mathscr{A}$ on $\Sigma^iM$ is
$$a\,(\Sigma^im) = (-1)^{|a|i}\Sigma^i(am),$$
for all graded element $ a\in \mathscr{A}$ and  $\Sigma^im\in \Sigma^iM$. And the
differential $\partial_{\Sigma^iM}$ of $\Sigma^iM$ is defined by
$$\partial_{\Sigma^iM}(\Sigma^im) = (-1)^i\Sigma^i\partial_{M}(m),$$
for all graded elements $m\in M$.

Let $Y$ be a subset of be a DG $\mathscr{A}$-module  $F$. It is called categorically free if for each DG module $\mathscr{A}$-module $P$ and any homogeneous map $g: Y\to P$ of degree $0$, there exists a unique morphism of DG $\mathscr{A}$-modules $f: F \to P$  such that $f|_Y=g$.
A DG $\mathscr{A}$-module $F$ is called categorically free if there exist a categorically free subset $Y\subset F$ such that $$
        F=\bigoplus_{y \in Y}[\mathscr{A}y\oplus\mathscr{A}\partial_{F}(y)].$$

Let $F$ be a DG $\mathscr{A}$-module. It is called called DG free, if it is isomorphic to a
direct sum of suspensions of ${}_{\mathscr{A}}\mathscr{A}$. Note that DG free $\mathscr{A}$-modules are not free objects in
the category of DG $\mathscr{A}$-modules. Let $Y$ be a graded set, we denote
$\mathscr{A}^{(Y)}$ as the DG free $\mathscr{A}$-module $\bigoplus\limits_{y\in Y}\mathscr{A} e_y$, where
$|e_y|=|y|$ and $\partial(e_y)=0$. A DG $\mathscr{A}$-module in $\mathscr{D}(\mathscr{A})$ is called ghost
projective if it is isomorphic in $\mathscr{D}(\mathscr{A})$ to a direct
summand of a certain DG free $\mathscr{A}$-module.

 Let $F$ be a DG $\mathscr{A}$-module and $E$ a subset of $F$. We say that $E$ is a semi-basis of $F$ and $F$ is a semi-free DG $\mathscr{A}$-module, if $E$ is a free basis of
$F^{\#}$ over $\mathscr{A}^{\#}$ and has a decomposition $E =
\bigsqcup_{i\ge0}E_i$ as a union of disjoint graded sets $E_u$ such
that
$$\partial(E_0)=0 \,\,\, \textrm {and} \,\,\,\partial(E_u)\subseteq \mathscr{A}
(\bigsqcup_{i<u}E_i)\, \,\,\textrm{for all}\,\, \,u >0.$$

Assume that $E$ is a semi-basis of a semi-free DG $\mathscr{A}$-module $F$ as above. Then we can obtain a sequence of DG $\mathscr{A}$-submodules $F(n)$, $n\in \mathbb{N}$, such that $$F(n)^{\#}=\bigoplus\limits_{i=0}^n\bigoplus\limits_{j\in E_i}\mathscr{A}^{\#}e_{i_j}.$$
Then
$$0=F(-1)\subset F(0)\subset\cdots\subset F(n)\subset\cdots$$
is a filtration of DG $\mathscr{A}$-submodules of $F$, such that $F = \bigcup\limits_{n=0}^{+\infty} \,F(n)$ and each $F(n)/F(n-1)$
is DG free on a basis of cocycles. We call the filtration above a semi-free filtration of $F$.  Conversely, if a DG $\mathscr{A}$-module $P$ admits a filtration of DG $\mathscr{A}$-submodules:
$$0=P(-1)\subset P(0)\subset P(1)\subset \cdots \subset P(i)\subset \cdots$$ such that $P=\bigcup\limits_{i=0}^{+\infty} \,P(i)$ and each $P(i)/P(i-1)$ is a DG free module on a cocycle basis $\Lambda_i=\{e_{j_i}|j\in I_i\}, i\in \mathbb{N}$. Then it is easy to check that $P$ is semi-free with $P^{\#}=\bigoplus\limits_{i=0}^{+\infty}\bigoplus\limits_{j\in I_i}\mathscr{A}^{\#}e_{j_i}$, and $\Lambda=\bigsqcup_{i\ge0}\Lambda_i$ is a semi-basis of $P$. Let $F$ be a semi-free DG $\mathscr{A}$-module. The DG free class of $F$ is
defined to be the number
$$ \inf\{n\in \Bbb{N}\cup \{0\}\,|\, F \,\text{admits a strictly
increasing semi-free filtration of length}\,\, n\}.$$ We denote it
by $\mathrm{DG free\,\, class}_{\mathscr{A}}F$.

\subsection{Morphisms between DG modules}
Let $M, N$ be two DG $\mathscr{A}$-modules. An $\mathscr{A}$-homomorphism $f:M \to N$
of degree $i$ is a $\k$-linear map of degree $i$ such that
\begin{align*}
f(am) = (-1)^{i\cdot|a|}af(m),\,\,\, \textrm{for all}\,\,\, a\in \mathscr{A},
m\in M.
\end{align*}
Denote $\Hom_{\mathscr{A}}(M,N)$ by the graded vector space of all graded
$\mathscr{A}$-homomorphisms from $M$ to $N$. This is a complex, with the
differential $\partial_{\Hom}$ defined as
\begin{align*}
\partial_{\Hom}(f)=\partial_{N}\circ f -
(-1)^{|f|}f\circ\partial_{M}
\end{align*}
for all $f\in \Hom_{\mathscr{A}}(M, N)$. A DG morphism $f:M \to N$ is an
$\mathscr{A}$-homomorphism of degree $0$ such that $\partial_{N}\circ f =
f\circ \partial_{M}$. The induced map $H(f)$ of $f$ on the
cohomologies is then a morphism of $H(\mathscr{A})$-modules.
A DG morphism $f: M\to N$ is called a quasi-isomorphism
(resp., ghost
morphism) if $H(f)$ is an isomorphism (resp., $H(f)=0$).
We write $f: M
\stackrel{\simeq}{\to}N$ if
$f$ is a
quasi-isomorphism.  A DG $\mathscr{A}$-module $M$
is called quasi-trivial if $M \simeq 0$.

\subsection{Resolutions of DG modules}
Let $P$ be a DG $\mathscr{A}$-module. If $\Hom_{\mathscr{A}}(P, -)$ preserves quasi-isomorphisms (resp. surjective quasi-isomorphisms), then $P$ is called homotopically projective (resp. semi-projective). If $\Hom_{\mathscr{A}}(P, -)$ turns surjective morphisms into surjective quasi-isomorphism, then $P$ is called categorically projective.
It is well known that any semi-projective (resp., categorically projective) DG $\mathscr{A}$-module is a direct summand of a semi-free (resp., categorically free) DG $\mathscr{A}$-module (cf. \cite{AFH}).

A semi-projective resolution
(resp., homotopically projective resolution) of a DG $\mathscr{A}$-module $M$ is a
quasi-isomorphism $\theta: P \to M$, where $P$ is a semi-projective
(resp., homotopically projective) DG $\mathscr{A}$-module.
A semi-free resolution of a DG $\mathscr{A}$-module $M$ is a
quasi-isomorphism $\varepsilon:F \to M$, where $F$ is a semi-free DG
$A$-module. Sometimes we call $F$ itself a semi-free resolution of
$M$. Every DG module has a semi-free resolution (\cite[Proposition
6.6]{FHT}). Similarly, a DG $\mathscr{A}$-module $I$ is called homotopically injective if
$\Hom_{\mathscr{A}}(-,I)$ preserves quasi-isomorphisms. A homotopically injective
resolution of a DG $\mathscr{A}$-module $M$ is a quasi-isomorphism $\eta:
M\to I$, where $I$ is a homotopically injective DG $\mathscr{A}$-module.  Every DG module
has a homotopically injective resolution (\cite[Theorem 3.7]{AFH}).

Two DG morphisms $\varphi, \psi: M \to N$ are called homotopic, if
there is an $\mathscr{A}$-homomorphism $h: M \to N$ of degree $-1$ such that
$\varphi - \psi = h \circ \partial_{M} + \partial_{N}\circ h$.
A DG morphism $\varphi: M\to N$ is called a homotopy
equivalence if there is a DG morphism $\psi: N \to M$ such that
$\varphi \circ \psi$ and $\psi \circ \varphi$ are homotopic to the
respective identity morphism.

\subsection{Cone of morphisms}
Let $f: M \to N$ be a DG morphism of DG $\mathscr{A}$-modules. The
mapping cone of $f$ is a DG $\mathscr{A}$-module $\mathrm{Cone}(f)$,
whose underlying graded module is $N^{\#} \oplus \Sigma M^{\#}  $.
And the differential of $\mathrm{Cone}(f)$ is defined by
$$\partial_{\mathrm{Cone}(f)}^i(n, \Sigma m) =
(f^{i+1}(m) + \partial ^i _N (n), -\Sigma \partial ^{i+1}_M (m)),$$
for all graded elements $(n, \Sigma m)$ of degree $i$, $i \in
\mathbb{Z}$. The mapping cone sequence
$$
0 \to   N \stackrel{\iota}{\to} \mathrm{Cone}(f) \stackrel{\pi}{\to}
\Sigma M\to 0
$$
where $\iota: n \mapsto\,(n,0)$ and $\pi: (n, \Sigma m) \mapsto
\Sigma m,$ is a linearly split short exact sequence of DG
$\mathscr{A}$-modules.

\subsection{Categories of DG modules}
The category of DG $\mathscr{A}$-modules is denoted by $\mathscr{C}(\mathscr{A})$ whose
morphisms are DG morphisms. We denote by $\mathscr{SF}(\mathscr{A})$ (resp. $\mathscr{SF}_{fg}(\mathscr{A})$) the full subcategory of $\mathscr{C}(\mathscr{A})$
whose objects are semi-free DG $\mathscr{A}$-modules (resp.
semi-free DG $\mathscr{A}$-modules with a finite semi-basis).
The homotopy category $\mathscr{K}(\mathscr{A})$ is the quotient category of
$\mathscr{C}(\mathscr{A})$, whose objects are the same as those of
$\mathscr{C}(A)$ and whose morphisms are the homotopic equivalence
classes of morphisms in $\mathscr{C}(\mathscr{A})$.
The derived category of DG $\mathscr{A}$-modules is denoted by
$\mathscr{D}(\mathscr{A})$, which is constructed from the category
$\mathscr{C}(\mathscr{A})$ by inverting quasi-isomorphisms
(\cite{Wei,KM}). A DG $\mathscr{A}$-module $M$ is compact if the functor $\Hom_{\mathscr{D}(\mathscr{A})}(M,-)$ preserves
all coproducts in $\mathscr{D}(\mathscr{A})$.
 By \cite[Proposition 3.3]{MW2},
a DG $\mathscr{A}$-module  is compact if and only if it admits a minimal semi-free resolution with a finite semi-basis. The full subcategory of $\mathscr{D}(\mathscr{A})$ consisting of compact DG $\mathscr{A}$-modules is denoted by $\mathscr{D}^c(\mathscr{A})$.

Let $\mathcal{C}$ be a subcategory or simply a set of some objects
of $\mathscr{D}(\mathscr{A})$. We denote by $\mathrm{smd}(\mathcal{C})$ the
minimal strictly full subcategory which contains $\mathcal{C}$ and
is closed under taking (possible) direct summands. And we write
$\overline{\mathrm{add}}(\mathcal{C})$ (resp.
$\mathrm{add}(\mathcal{C})$) as the intersection of all strict and
full subcategories of $\mathscr{D}(\mathscr{A})$ that contain $\mathcal{C}$ and
are closed under direct sums (resp. finite direct sums) and all
suspensions.

Let $\mathcal{S}$ and $\mathcal{T}$ be two strict and full
subcategories of $\mathscr{D}(\mathscr{A})$. We define $\mathcal{S}\star
\mathcal{T}$ as a full subcategory of $\mathscr{D}(\mathscr{A})$,  whose
objects are described as follows: $M\in \mathcal{S}\star
\mathcal{T}$ if and only if there is an exact triangle
$$ L\to M\to N\to \Sigma L, $$
where $L\in \mathcal{S}$ and $N\in \mathcal{T}$. For any strict and
full subcategory $\mathcal{R}$ of $\mathscr{D}(\mathscr{A})$, one has
$\mathcal{R}\star (\mathcal{S}\star \mathcal{T}) = (\mathcal{R}\star
\mathcal{S})\star \mathcal{T}$ (see \cite{BBD} or \cite[1.3.10]{BVDB}). Thus, the
following notation is unambiguous:
\begin{equation*}
\mathcal{T}^{\star n} =
\begin{cases} 0\quad&\text{for}\, n =0;\\
\mathcal{T}\quad & \text{for}\, n =1; \\
\overbrace{\mathcal{T}\star\cdots\star
\mathcal{T}}^{n\,\text{copies}} \quad&\text{for}\, n\ge 1.
\end{cases}
\end{equation*}
We refer to the objects of $\mathcal{T}^{\star n}$ as $(n-1)$-fold
extensions of objects from $\mathcal{T}$. Define $\mathcal{S}\diamond\mathcal{T} =
\mathrm{smd}(\mathcal{S}\star\mathcal{T})$. Inductively, we define
$$\langle\mathcal{S}\rangle_1 =
\mathrm{smd}(\mathrm{add}(\mathcal{S})), \text{and}\,
\langle\mathcal{S}\rangle_n =
\langle\mathcal{S}\rangle_{n-1}\diamond \langle\mathcal{S}\rangle_1,
n\ge 2.$$ Similarly,  we define
$\overline{\langle\mathcal{S}\rangle}_1 =
\mathrm{smd}(\overline{\mathrm{add}}(\mathcal{S})),
\overline{\langle\mathcal{S}\rangle}_n =
\overline{\langle\mathcal{S}\rangle}_{n-1}\diamond
\overline{\langle\mathcal{S}\rangle}_1, n\ge 2$. We have the
associativity of $\diamond$ and the formula
$$\mathcal{C}_1\diamond\mathcal{C}_2\diamond\cdots\diamond
\mathcal{C}_n
=\mathrm{smd}(\mathcal{C}_1\star\cdots\star\mathcal{C}_n)$$ (see
\cite[Section 2]{BVDB}). Denote $\langle\mathcal{S}\rangle=
\bigcup_{i\ge 0}\langle\mathcal{S}\rangle_i.$

\section{quillen-suslin theorem in dg version}
In this section, we will prove that any finitely generated semi-projective (resp., categorically projective) DG $\mathscr{A}$-module is semi-free (resp., categorically free). For this, we should recall several useful lemmas.
\begin{lem}\cite[Remark $20.1$]{FHT}\label{basic}
Let $A$ be a positively graded algebra over a commutative ring $R$ with $A^0=R$. Then any bounded below projective graded $A$-module is free.
\end{lem}

\begin{lem}\cite[Lemma $2.8$]{Mao2}\label{toiso} Let $\mathscr{A}$ be a connected cochain DG algebra. Suppose that $F$ is a bounded below DG $\mathscr{A}$-module such that $\partial_F(F)\subseteq \frak{m}F$ and $F^{\#}$ is a projective $\mathscr{A}^{\#}$-module. If a DG morphism $\alpha: F\to F$ is homotopic to the identity morphism $\mathrm{id}_F$, then $\alpha$ is an isomorphism.
\end{lem}

\begin{lem}\cite[Lemma $2.11$]{Mao2}\label{useful} Let $\mathscr{A}$ be a connected cochain DG algebra. Assume that $\mathscr{A}$ is a DG free $\mathscr{A}$-module with a direct summand $P$ such that $H(P)$ is bounded below. Then $P$ is also a DG free $\mathscr{A}$-module.
\end{lem}

\begin{lem}\cite[Lemma $3.4$]{Mao2} \label{cone} Let $F$ be a semi-free DG $\mathscr{A}$-module and let $F'$ be a semi-free DG submodule of $F$ such that $F/F'=\mathscr{A}\otimes V$ is DG free on a set of cocycles. Then there exsits a DG morphism $f:\mathscr{A}\otimes \Sigma^{-1}V\to F'$ such that $F=\mathrm{cone}(f)$.
\end{lem}

\begin{lem}\label{heq}\cite[Proposition $6.4$(iii)]{FHT}
Assume that $M$ and $N$ are semi-free DG $\mathscr{A}$-modules and $f: M\to N$  is a quasi-isomorphism. Then $f$ is a homotopic equivalence.
\end{lem}

\begin{thm}\label{semiqst}
 Let $\mathscr{A}$ be a connected cochain DG algebra and $P$ a DG $\mathscr{A}$-module such that its underlying graded module $P^{\#}$ is a finitely generated $\mathscr{A}^{\#}$-module. Then $P$ is semi-free if it is semi-projective.
\end{thm}

\begin{proof}
Since $P$ is semi-projective,  it is a direct summand of a semi-free DG $\mathscr{A}$-module $F$. We have $F=P\oplus Q$, for some DG $\mathscr{A}$-module $Q$. Then $F^{\#}=P^{\#}\oplus Q^{\#}$ as a graded $\mathscr{A}^{\#}$-module. It implies that $P^{\#}$ is a projective $\mathscr{A}^{\#}$-module. Since $P^{\#}$ is finitely generated, it is bounded below. By Lemma \ref{basic}, $P^{\#}$ is a free $\mathscr{A}^{\#}$-module.

If $\mathrm{DGfree.class}_{\mathscr{A}}F=t$ is finite,  then
$F$ admits a strictly increasing semi-free filtration:
$$0=F(-1)\subset F(0)\subset F(1)\subset  \cdots \subset F(t-1)\subset
F(t)=F, $$
where each quotient $F(i)/F(i-1)$ is a DG free $\mathscr{A}$-module.  Let $E_i=\{e_{i_j}|j\in I_i\}$ be a DG free basis of $F_i=F(i)/F(i-1), 0\le i\le t$. We have $\partial_F(e_{i_j})\subseteq \mathscr{A}(\bigsqcup\limits_{j=0}^{i-1}E_j)$.
 Then each graded free $\mathscr{A}^{\#}$-module $F(r)^{\#}$ can be decomposed as
 $$\bigoplus\limits_{i=0}^r\bigoplus\limits_{j\in I_i}\mathscr{A}^{\#}e_{i_j}, 0\le r\le t.$$
Let $e_{i_j}=p_{i_j}+q_{i_j}$, where $p_{i_j}\in P$ and $q_{i_j}\in Q$, for any $j\in I_i, i=0,1,\cdots,t$. We have
$$F(r)=[\sum\limits_{i=0}^r(\sum\limits_{j\in I_i}\mathscr{A}p_{i_j})]\oplus [\sum\limits_{i=0}^r(\sum\limits_{j\in I_i}\mathscr{A}q_{i_j})], r=0,1,\cdots, t.$$
Hence $$F_r=F(r)/F(r-1)=(\sum\limits_{j\in I_r}\mathscr{A}\overline{p_{i_j}})\oplus(\sum\limits_{j\in I_r}\mathscr{A}\overline{q_{r_j}}), r=0, 1,\cdots, t.$$
By Lemma \ref{useful}, each $\sum\limits_{j\in I_r}\mathscr{A}\overline{p_{i_j}}$ is either a zero module or a DG free $\mathscr{A}$-module.
Let $\omega_{r_{\lambda}},\lambda\in \Lambda_r$ be its DG free basis ($\Lambda_r=\emptyset$ and $\omega_{r_{\lambda}}=0$ if $\sum\limits_{j\in I_r}\mathscr{A}\overline{p_{r_j}}=0$). Then $$\sum\limits_{j\in I_r}\mathscr{A}\overline{p_{r_j}}=\bigoplus\limits_{j\in \Lambda_r}\mathscr{A}\omega_{r_j}.$$ Note that
$$\bigoplus\limits_{i\in I_r}\mathscr{A}^{\#}e_{r_i}=\bigoplus\limits_{\lambda\in \Lambda_r}\mathscr{A}^{\#}\omega_{r_{\lambda}}\oplus (\sum\limits_{j\in I_r}\mathscr{A}\overline{q_{r_j}})^{\#}$$ is a graded $\mathscr{A}^{\#}$-submodule of $F^{\#}$. So $\bigoplus\limits_{\lambda\in \Lambda_r}\mathscr{A}^{\#}\omega_{r_{\lambda}}$ is also a graded $\mathscr{A}^{\#}$-submodule of $F^{\#}$. Since $\partial_{F_r}(\omega_{r_{\lambda}})=0$, we have
$$\partial_F(\omega_{r_{\lambda}})\in F(r-1)\cap \sum\limits_{i=0}^r(\sum\limits_{j\in I_i}\mathscr{A}p_{i_j})=\sum\limits_{i=0}^{r-1}(\sum\limits_{j\in I_i}\mathscr{A}p_{i_j}),   r=0,1,\cdots, t.$$
Let $P(r)=\sum\limits_{i=0}^r(\sum\limits_{j\in I_i}\mathscr{A}p_{i_j}), r=0,1,\cdots,t$. Then
$$0\subseteq P(0)\subseteq P(1)\subseteq P(2)\subseteq \cdots \subseteq P(t-1)\subseteq P(t)=P$$ is a filtration of DG $\mathscr{A}$-submodules of $P$. Moreover, $P(r)/P(r-1)=\sum\limits_{j\in I_r}\mathscr{A}\overline{p_{r_j}}$ is either zero or a DG free $\mathscr{A}$-module $\bigoplus\limits_{j\in \Lambda_r}\mathscr{A}\omega_{r_j}$. If $P(r)/P(r-1)=0$, for some $r\in \mathbb{N}$, then we just cancel such $P(r)$. In this way, we can get a strictly increasing semi-free filtration of $P$. So $P$ is semi-free in this case.

If $\mathrm{DGfree.class}_{\mathscr{A}}F = +\infty$, then
$F$ has a strictly increasing semi-free filtration:
$$0=F(-1)\subset F(0)\subset F(1)\subset  \cdots \subset F(i)\subset
\cdots, $$
such that $F=\bigcup\limits_{i=0}^{+\infty}F(i)$ and each quotient $F(i)/F(i-1)$ is a DG free $\mathscr{A}$-module. Similarly as above, let $E_i=\{e_{i_j}|j\in I_i\}$ be a free basis of $F_i=F(i)/F(i-1), i\in \mathbb{N}$. We have $\partial_F(e_{i_j})\subseteq \mathscr{A}(\bigsqcup\limits_{j=0}^{i-1}E_j)$.
 Then each graded free $\mathscr{A}^{\#}$-module $F(r)^{\#}$ can be decomposed as
 $$\bigoplus\limits_{i=0}^r\bigoplus\limits_{j\in I_i}\mathscr{A}^{\#}e_{i_j},  r\in \mathbb{N}.$$
 Let $e_{i_j}=p_{i_j}+q_{i_j}$, where $p_{i_j}\in P$ and $q_{i_j}\in Q$, for any $j\in I_i, i\in \mathbb{N}$. We have
$$F(r)=[\sum\limits_{i=0}^r(\sum\limits_{j\in I_i}\mathscr{A}p_{i_j})]\oplus [\sum\limits_{i=0}^r(\sum\limits_{j\in I_i}\mathscr{A}q_{i_j})], r\in \mathbb{N}.$$
Hence $$F_r=F(r)/F(r-1)=(\sum\limits_{j\in I_r}\mathscr{A}\overline{p_{i_j}})\oplus(\sum\limits_{j\in I_r}\mathscr{A}\overline{q_{r_j}}), r\in \mathbb{N}.$$
By Lemma \ref{useful}, each $\sum\limits_{j\in I_r}\mathscr{A}\overline{p_{i_j}}$ is either a zero module or a DG free $\mathscr{A}$-module.
Let $\omega_{r_{\lambda}},\lambda\in \Lambda_r$ be its DG free basis ($\Lambda_r=\emptyset$ and $\omega_{r_{\lambda}}=0$ if $\sum\limits_{j\in I_r}\mathscr{A}\overline{p_{r_j}}=0$). Then $$\sum\limits_{j\in I_r}\mathscr{A}\overline{p_{r_j}}=\bigoplus\limits_{j\in \Lambda_r}\mathscr{A}\omega_{r_j}.$$ Note that
$$\bigoplus\limits_{i\in I_r}\mathscr{A}^{\#}e_{r_i}=\bigoplus\limits_{\lambda\in \Lambda_r}\mathscr{A}^{\#}\omega_{r_{\lambda}}\oplus (\sum\limits_{j\in I_r}\mathscr{A}\overline{q_{r_j}})^{\#}$$ is a graded $\mathscr{A}^{\#}$-submodule of $F^{\#}$. So $\bigoplus\limits_{\lambda\in \Lambda_r}\mathscr{A}^{\#}\omega_{r_{\lambda}}$ is also a graded $\mathscr{A}^{\#}$-submodule of $F^{\#}$. Since $\partial_{F_r}(\omega_{r_{\lambda}})=0$, we have
$$\partial_F(\omega_{r_{\lambda}})\in F(r-1)\cap \sum\limits_{i=0}^r(\sum\limits_{j\in I_i}\mathscr{A}p_{i_j})=\sum\limits_{i=0}^{r-1}(\sum\limits_{j\in I_i}\mathscr{A}p_{i_j}),   r\in \mathbb{N}.$$
Let $P(r)=\sum\limits_{i=0}^r(\sum\limits_{j\in I_i}\mathscr{A}p_{i_j}), r\in \mathbb{N}$. Then we get the following increasing sequence:
$$0\subseteq P(0)\subseteq P(1)\subseteq P(2)\subseteq \cdots \subseteq P(i)\subseteq \cdots $$
of DG $\mathscr{A}$-submodules of $P$.
Moreover, $P(r)/P(r-1)=\sum\limits_{j\in I_r}\mathscr{A}\overline{p_{r_j}}$ is either zero or a DG free $\mathscr{A}$-module $\bigoplus\limits_{j\in \Lambda_r}\mathscr{A}\omega_{r_j}$. If $P(r)/P(r-1)=0$, for some $r\in \mathbb{N}$, then we just cancel such $P(r)$.
Note that
\begin{align*}
\bigcup\limits_{r=0}^{+\infty}P(r)^{\#}&=\sum\limits_{i=0}^{+\infty}\sum\limits_{j\in I_i}\mathscr{A}^{\#}p_{ij}=\bigoplus\limits_{r=0}^{+\infty}\bigoplus\limits_{\lambda\in \Lambda_r}\mathscr{A}^{\#}\omega_{r_{\lambda}}\\
&= \frac{\bigoplus\limits_{r=0}^{+\infty}\bigoplus\limits_{i\in I_r}\mathscr{A}^{\#}e_{r_i}}{\sum\limits_{i=0}^{+\infty}\sum\limits_{j\in I_i}\mathscr{A}^{\#}q_{ij}} \supseteq \frac{F^{\#}}{Q^{\#}}=P^{\#}.\\
\end{align*}
Then $\bigcup\limits_{r=0}^{+\infty}P(r)=P$ since $\bigcup\limits_{r=0}^{+\infty}P(r)\subseteq P$ is obviously.
So $P$ has a strictly increasing semi-free filtration and it is also semi-free.
\end{proof}
\begin{rem}{\rm
By \cite[Lemma $3.2$]{MW1}, the triangulated category $\mathscr{K}(\mathscr{SF}_{fg}(\mathscr{A}))$ is equivalent to
$\mathscr{D}^c(\mathscr{A})$, which is the smallest
triangulated thick subcategory of $\mathscr{D}(\mathscr{A})$ containing ${}_{\mathscr{A}}\mathscr{A}$
(see \cite[Theorem 5.3]{Kel}). Hence $\mathscr{K}(\mathscr{SF}_{fg}(\mathscr{A}))$ is closed under direct summands.
 Theorem \ref{semiqst} indicates that the category $\mathscr{SF}_{fg}(\mathscr{A})$ is also closed under direct summands. }
\end{rem}

\begin{cor}
Let $\mathscr{A}$ be a connected cochain DG algebra and $P$ a DG $\mathscr{A}^e$-module such that its underlying graded module $P^{\#}$ is a finitely generated $\mathscr{A}^{e\#}$-module. Then $P$ is semi-free if it is semi-projective.
\end{cor}

\begin{thm}
Let $\mathscr{A}$ be a connected cochain DG algebra and $P$ a DG $\mathscr{A}$-module such that its underlying graded module $P^{\#}$ is a finitely generated $\mathscr{A}^{\#}$-module. Then $P$ is a categorically free DG $\mathscr{A}$-module if it is categorically projective.
\end{thm}

\begin{proof}
Since $P$ is categorically projective, it is a direct summand of a categorically free DG $\mathscr{A}$-module $F$. There exists a DG $\mathscr{A}$-module $Q$ such that $F=P\oplus Q$. Let $Y$ be a categorically free subset of $F$. Then $$F=\bigoplus_{y\in Y}[\mathscr{A}y\oplus \mathscr{A}\partial_F(y)].$$  Let $F(0)=\bigoplus\limits_{y\in Y}\mathscr{A}\partial_F(y)$ and $F(1)=F$. Then $0\subset F(0)\subset F(1)=F$ is a strictly increasing semi-free filtration of $F$.  Let $\partial_F(y)=p_{0_y}+q_{0_y}$ and $y=p_{1_y}+q_{1_y}, y\in Y$, where $p_{0_y},p_{1_y}\in P$ and $q_{0_y},q_{1_y}\in Q$.
Then $$F(0)=\bigoplus_{y\in Y}\mathscr{A}\partial_F(y)=\sum\limits_{y\in Y}\mathscr{A}p_{0_y}\oplus\sum\limits_{y\in Y}\mathscr{A}q_{0_y}$$
and $$F(1)/F(0)=\bigoplus_{y\in Y}\mathscr{A}y=\sum\limits_{y\in Y}\mathscr{A}p_{1_y}\oplus\sum\limits_{y\in Y}\mathscr{A}q_{1_y}.$$
By Lemma \ref{useful}, $\sum\limits_{y\in Y}\mathscr{A}p_{0_y}$ and $\sum\limits_{y\in Y}\mathscr{A}p_{1_y}$ are  DG free $\mathscr{A}$-modules.
Let $\{e_{\lambda}|\lambda\in \Lambda\}$ and $\{\varepsilon_{\omega}|\omega\in \Omega\}$ be their DG free basis, respectively. Then $$\sum\limits_{y\in Y}\mathscr{A}p_{0_y}=\bigoplus_{\lambda\in \Lambda}\mathscr{A}e_{\lambda}\quad \text{and}\quad \sum\limits_{y\in Y}\mathscr{A}p_{1_y}=\bigoplus_{\omega\in \Omega}\mathscr{A}\varepsilon_{\omega}.$$
Let $P(0)=\sum\limits_{y\in Y}\mathscr{A}p_{0_y}$ and $P(1)=P$. Then
$$\frac{P(1)}{P(0)}=\frac{F(1)}{P(0)\oplus Q}=\frac{F(1)/F(0)}{\sum\limits_{y\in Y}\mathscr{A}q_{1_y}}=\sum\limits_{y\in Y}\mathscr{A}p_{1_y}=\bigoplus_{\omega\in \Omega}\mathscr{A}\varepsilon_{\omega}.$$
So $0=P(-1)\subset P(0)\subset P(1)=P$ is a semi-free filtration of $P$. Define a DG morphism $$\beta: \Sigma^{-1}[P(1)/P(0)]=\bigoplus_{\omega\in \Omega}\mathscr{A}\Sigma^{-1}\varepsilon_{\omega}\to P(0)$$ such that $\beta(\Sigma^{-1}\varepsilon_{\omega})=\partial_P(\varepsilon_{\omega})$.
By the proof of Lemma \ref{cone}, we have $P(1)=\mathrm{cone}(\beta)$.
On the other hand, the short exact sequence of DG $\mathscr{A}$-modules:
$$0\to P(0)\stackrel{\iota}{\to} P(1)\stackrel{\pi}{\to} P(1)/P(0)\to 0$$
induces a long exact sequence of cohomologies
$$\cdots \stackrel{\delta^{i-1}}{\to}H^i(P(0))\stackrel{H^i(\iota)}{\to} H^i[P(1)]\stackrel{H^i(\pi)}{\to} H^i[P(1)/P(0)]\stackrel{\delta^i}{\to}H^{i+1}(P(0))\stackrel{H^{i+1}(\iota)}{\to}\cdots.$$
Since $H(P)=0$, we have $H^i(\iota)=0, \forall i\in \mathbb{Z}$. The long exact sequence above implies that each connecting homomorphism $\delta^i$ is an isomorphism. Note that $$H[P(1)/P(0)]=\bigoplus\limits_{\omega\in \Omega}H(\mathscr{A})\lceil \varepsilon_{\omega}\rceil.$$ By the definition of connecting homomorphism, we have $\delta^{|\varepsilon_{\omega}|}(\lceil \varepsilon_{\omega}\rceil)=\lceil\partial_P(\varepsilon_{\omega})\rceil$.
Since $H(\beta)(\lceil \Sigma^{-1}\varepsilon_{\omega}\rceil)=\lceil \partial_P(\varepsilon_{\omega})\rceil$, we conclude that $\beta$ is a quasi-isomorphism.
By Lemma \ref{heq}, $\beta$ is homotopic equivalence. So there exists a DG morphism $$\alpha: P(0)\to \Sigma^{-1}[P(1)/P(0)]$$ such that
$\alpha\circ \beta \simeq \mathrm{id}_{\Sigma^{-1}[P(1)/P(0)]}$ and $\beta\circ \alpha\simeq \mathrm{id}_{P(0)}$.
Applying Lemma \ref{toiso}, we get $\alpha\circ \beta$ and $\beta\circ \alpha$ are both isomorphism. Thus $\beta$ is bijective.
Since $\beta$ is surjective, we have $P(0)=\sum\limits_{\omega\in \Omega}\mathscr{A}\partial_P(\varepsilon_{\omega})$. We claim that $\partial_{P}(\varepsilon_{\omega}), \omega\in \Omega$ are $\mathscr{A}$-linearly independent.
Indeed, if the $\mathscr{A}$-linear sum $\sum\limits_{\omega\in \Omega}a_{\omega}\partial_P(\varepsilon_{\omega})=0$ in $P(0)$, then $$0=\beta^{-1}(0)=\beta^{-1}[\sum\limits_{\omega\in \Omega}a_{\omega}\partial_P(\varepsilon_{\omega})]=\sum\limits_{\omega\in \Omega}a_{\omega}\Sigma^{-1}\varepsilon_{\omega}.$$
Since $\Sigma^{-1}\varepsilon_{\omega},\omega\in \Omega$ is a DG free basis of $\Sigma^{-1}P(1)/P(0)$, we get $a_{\omega}=0, \omega\in \Omega$. Then $P(0)=\bigoplus\limits_{\omega\in \Omega}\mathscr{A}\partial_{P}(\varepsilon_{\omega})$ and $P(1)=\mathrm{cone}(\beta)=\bigoplus\limits_{\omega\in \Omega}[\mathscr{A}\varepsilon_{\omega}\oplus \mathscr{A}\partial_P(\varepsilon_{\omega})]$. Therefore, $P=P(1)$ is a categorically free DG $\mathscr{A}$-module.

\end{proof}
\begin{cor}
Let $\mathscr{A}$ be a connected cochain DG algebra and $P$ a DG $\mathscr{A}^e$-module such that its underlying graded module $P^{\#}$ is a finitely generated $\mathscr{A}^{e\#}$-module. Then $P$ is a categorically free DG $\mathscr{A}^e$-module if it is categorically projective.

\end{cor}

In the rest of this section, we want to illustrate the relations between the classical Quillen-Suslin Theorem and Theorem \ref{semiqst}. In \cite{MGYC}, the notion of DG polynomial algebra was introduced and systematically studied. Recall that a connected cochain DG algebra $\mathscr{A}$ 
is called a DG polynomial algebra if $\mathscr{A}^{\#}=\k[x_1,x_2,\cdots,x_n]$ is a polynomial graded algebra with each $|x_i|=1$. By \cite[Theorem 3.1]{MGYC}, there exist some $t_1,t_2,\cdots,t_n\in \k$ such that $\partial_{\mathscr{A}}$ is defined  by
$$\partial_{\mathscr{A}}(x_i)=\sum\limits_{j=1}^nt_jx_ix_j=\sum\limits_{j=1}^{i-1}t_jx_jx_i+t_ix_i^2+\sum\limits_{j=i+1}^nt_jx_ix_j, \text{for any} \,\, i\in \{1,2,\cdots,n\}.$$
We write $\mathscr{A}(t_1,t_2,\cdots, t_n)$ for the DG polynomial algebra mentioned above. 
The set of DG polynomial algebras in $n$ degree one variables is
$$\Omega(x_1,x_2,\cdots,x_n)=\{\mathscr{A}(t_1,t_2,\cdots,t_n)|t_i\in
\k, i=1,2,\cdots, n\}\cong \Bbb{A}_{\k}^n.$$
 By \cite[Corollary 4.2]{MGYC}, 
there are only two isomorphism
 classes $\mathscr{A}(0,0,\cdots,0)$ and $\mathscr{A}(1,0,\cdots, 0)$ in space $\Omega(x_1,x_2,\cdots, x_n)$. Since any DG polynomial algebra $\mathscr{A}(t_1,t_2,\cdots,t_n)$ is a connected cochain DG algebra, we have the following corollary. 

\begin{cor}\label{dgpoly}
Each finitely generated semi-projective (resp. categorically projective) DG module over DG polynomial algebra $\mathscr{A}(t_1,t_2,\cdots,t_n)$ 
is a semi-free (resp. categorically free) DG module. 
\end{cor}

\begin{rem} Let $P$ a finitely generated graded module over the graded polynomial algebra $\k[x_1,x_2,\cdots,x_n]$. We write $(P,\partial_P=0)$ for the DG module over the trivial DG polynomial algebra $\mathscr{A}(0,0,\cdots,0)$. It is easy to check that $P$ is projective (resp. free ) if and only if $(P, \partial_P=0)$ is a semi-projective (resp. semi-free) DG module over $\mathscr{A}(0,0,\cdots,0)$. In this sense, we can recover the original Quillen-Suslin Theorem by Corollary \ref{dgpoly}.

\end{rem}

\section{applications in invariants of dg modules}
In this section, we will give applications Quillen-Suslin Theorem in the studies of invariants for DG modules. We will reveal the relations between cone length, ghost length and level of DG modules over a connected cochain DG algebras.
\begin{defn}{\rm
Let $M$ be a non-quasi-trivial DG $\mathscr{A}$-module. The cone length of $M$ is defined to be
the number
$$\mathrm{cl}_{\mathscr{A}}M =
\inf\{\mathrm{\,DGfree\,\,class}_{\mathscr{A}}F\,|\,F \stackrel{\simeq}\to M
 \ \text{is a semi-free resolution of}\  M\}.$$ And we define $\mathrm{cl}_{\mathscr{A}} N=-1$ if $H(N)=0$. }
\end{defn}

The invariant cone length of a DG modules was introduced and studied in \cite{MW3}.
One sees that $\mathrm{cl}_{\mathscr{A}}M$ is just the invariant `$l(M)$' defined in \cite[Definition 9]{Car}
when $\mathscr{A}$ is a DG polynomial algebra with zero differential.
Note that $\mathrm{cl}_{\mathscr{A}}M$ may be $+\infty$. By the existence of
 Eilenberg-Moore resolution, we have
 $\mathrm{cl}_{\mathscr{A}}M \le \mathrm{pd}_{H(\mathscr{A})}H(M)$.

\begin{prop}\cite[Theorem 3.7]{Mao2}\label{clfree}
Let $M$ be an object in $\mathscr{D}^{+}(\mathscr{A})$ such that $\mathrm{cl}_{\mathscr{A}}M<\infty$, then there is a minimal semi-free resolution $G$ of $M$ such that $$\mathrm{DGfree.class}_{\mathscr{A}}G=\mathrm{cl}_{\mathscr{A}}M.$$
\end{prop}

\begin{prop}\label{conelen}
Let $M$ be a compact left DG module over a connected cochain DG algebra $\mathscr{A}$.  Then
$\mathrm{cl}_{\mathscr{A}}M=n$ if and only if $M\in \mathrm{add}(\mathscr{A})^{\star n+1}\setminus\mathrm{add}(\mathscr{A})^{\star n}$.
\end{prop}

\begin{proof}
We only need to prove  $M\in \mathrm{add}(\mathscr{A})^{\,\star n+1}$ if and only if
$\mathrm{cl}_{\mathscr{A}}M \le n$.

If $\mathrm{cl}_{\mathscr{A}}M =l \le n$, then  $M$ admits a minimal semi-free resolution
$F$ with $\mathrm{DGfree.class}_{\mathscr{A}}F=l$ by Proposition \ref{clfree}. Since $M\in \mathscr{D}^c(\mathscr{A})$, $F$ has a finite semi-basis.
The semi-free DG module $F$ has a strictly increasing semi-free filtration
$$0=F(-1)\subset F(0)\subset\cdots\subset F(i)\subset
F(i+1)\subset \cdots \subset F(l) = F$$ of length $l$. This yields a
sequence of short exact sequences
$$0\to F(i-1)\to F(i) \to F(i)/F(i-1)\to 0,\,  1\le i\le l.$$
Since $F(0)$ and $F(i)/F(i-1), \,i\ge 1$ are in
$\mathrm{add}(\mathscr{A})$, induction shows that $F$ is in
$\mathrm{add}(\mathscr{A})^{\,\star l+1}$. Hence $M$ is an object
in $\mathrm{add}(\mathscr{A})^{\,\star n+1}$.

Conversely,  let $M$ be an object in
$\mathrm{add}(\mathscr{A})^{\,\star n+1}$ for some $n\ge 0$. We
will prove $\mathrm{cl}_{\mathscr{A}}M \le n$ by induction on $n$. For $n=0$,
the assertion is evident. For $n\ge 1$, there exists an exact
triangle
$$ L\stackrel{\varepsilon}{\to} N\to M \to \Sigma L,$$
where $L$ and $N$ are in $\mathrm{add}(\mathscr{A})$ and
$\mathrm{add}(\mathscr{A})^{\,\star n}$ respectively. Since $L$ is
in $\mathrm{add}(\mathscr{A})$, the DG morphism $\varepsilon$ can
be represented by a DG morphism $\varepsilon': F_L \to F_N$, where
$F_L$ is a finite direct sum of shifted copies of $\mathscr{A}$ and $F_N$ is a minimal
semi-free resolution of $N$ with $\mathrm{DGfree.class}_{\mathscr{A}}F_N\le n-1$ by inductive assumption and Proposition \ref{clfree}.
Let $t=\mathrm{DGfree.class}_{\mathscr{A}}F_N$. Then $F_N$ admits
a strictly increasing semi-free
filtration
$$0=F_N(-1)\subset F_N(0)\subset \cdots\subset
F_N(i)\subset \cdots \subset F_N(t)=F_N$$ of length $t$. Clearly,
$M\cong \mathrm{Cone}(\varepsilon')$ in $\mathscr{D}(\mathscr{A})$. We have the
following cone exact sequence
$$ 0\to F_N\stackrel{\phi}{\to} \mathrm{Cone}(\varepsilon')
\to \Sigma F_L\to 0.$$ Since $\phi$ is an injective DG morphism,
$$0\subset \phi(F_N(0)) \subset \cdots \subset \phi(F_N(i))\subset
\cdots \subset \phi(F_N(t))\subset \mathrm{Cone}(\varepsilon')$$ is
a strictly semi-free filtration of $\mathrm{Cone}(\varepsilon')$ of
length $t+1\le n$. Since $\mathrm{Cone}(\varepsilon') \cong M$ in
$\mathscr{D}(\mathscr{A})$, we have $\mathrm{cl}_{\mathscr{A}}M \le n$.

\end{proof}

The concept of `ghost length' was first introduced by Christensen in \cite{Chr}.
Later, Hovey-Lockridge \cite{HL1} and  Kuribayashi
\cite{Kur1}
 applied this invariant to complexes and DG modules, respectively.
 \begin{defn}{\rm
A DG $\mathscr{A}$-module $M$ is said to have ghost
length $n$, written by $\mathrm{gh.len}_{\mathscr{A}} M = n$, if every composite
$$M\stackrel{f_1}{\to} I_1\stackrel{f_2}{\to} \cdots
\stackrel{f_{n+1}}{\to} I_{n+1}$$ of $n+1$ ghosts is $0$ in
$\mathscr{D}(\mathscr{A})$, and there is a composite of $n$ ghosts from $M$
that is not $0$ in $\mathscr{D}(\mathscr{A})$.
We set $\mathrm{gh.len}_{\mathscr{A}}M=-1$
if $M$ is zero object in $\mathscr{D}(\mathscr{A})$. }
 \end{defn}

\begin{lem}\label{homologysurj}
For any DG $\mathscr{A}$-module $M$, there is a DG free $\mathscr{A}$-module $F$ and a
morphism of DG $\mathscr{A}$-module $f: F\to M$, such that $H(f)$ is
surjective. Furthermore, if $M\in \mathscr{D}_{fg}(\mathscr{A})$, then $F$ has a finite DG free basis.
\end{lem}
\begin{proof}
Choose a set of cycles $Z$ in $M$ such that $\{[z]\,|\,z\in Z\}$
generate the graded $H(\mathscr{A})$-module $H(M)$. Let $E=\{e_z\,|\,z\in Z\}$
be a linearly independent set over $\mathscr{A}$. Define a DG free $\mathscr{A}$-module
$$F=\bigoplus\limits_{e_z\in E} \mathscr{A}e_z $$ by $\partial_F(e_z)=0$, for
any $e_z\in E$. Let $f: F\to M$ be the DG morphism defined by
$f(e_z) = z$, for any $z\in Z$. It is easy to check that $H(f)$ is
surjective.
\end{proof}

\begin{lem}\label{dgfree}
Assume that $f: F\to M$ is a ghost morphism of DG $\mathscr{A}$-modules
and $F$ is a DG free $\mathscr{A}$-module. Then $f$ is homotopic to zero.
\end{lem}

\begin{proof}
By the assumption that $F$ is a DG free $\mathscr{A}$-module, we denote it by $$F=\bigoplus\limits_{\lambda\in \Lambda}
\mathscr{A}e^{\lambda},$$ where $\partial_F(e^{\lambda})=0$ for any
$\lambda\in \Lambda$. Since $f$ is a DG morphism, we have
$\partial_M\circ f = f\circ
\partial_F$. Hence $f(e^{\lambda})\in \mathrm{Ker}(\partial_M)$, for any $\lambda\in \Lambda$. By
the assumption $H(f)=0$, we have $[f(e^{\lambda})]$ is zero in
$H(M)$. There exists $m^{\lambda}\in M$ such that $\partial_M(m^{\lambda})=
f(e^{\lambda})$. We define an $\mathscr{A}$-linear map $\sigma: F\to M$ by
$\sigma(e^{\lambda})=m^{\lambda}$ for any $\lambda\in \Lambda$. It
is easy to check that $f = \partial_M\circ\sigma + \sigma\circ
\partial_F$. Hence $f$ is homotopic to zero.
\end{proof}

\begin{lem}\label{summand}
Assume that $g: G \to M$ is a ghost morphism of DG $\mathscr{A}$-modules
and $G$ is a direct summand of a DG free $\mathscr{A}$-module. Then $g$ is
homotopic to zero.
\end{lem}

\begin{proof}
Suppose that $G\oplus H$ is a DG free $\mathscr{A}$-module. Let $f: G\oplus
H\to M$ be a DG morphism defined by $f|_{G}=g$ and $f|_{H}=0$. We
have $H(f)=0$. By Lemma \ref{dgfree}, $f$ is homotopic to $0$. There
is an $\mathscr{A}$-linear map $\sigma: G\oplus H\to M$ of degree $+1$ such
that $f =
\partial_M\circ\sigma + \sigma\circ
\partial_F$. Hence $f|_{G} = \partial_M\circ\sigma|_{G} + \sigma|_{G}\circ
\partial_F|_{G}$. That is $g =\partial_M\circ\sigma|_{G} + \sigma|_{G}\circ
\partial_G$ and $g$ is homotopic to zero.
\end{proof}

\begin{lem}\label{ghostproj}
Let $\gamma : M\to N$ be a ghost morphism in $\mathscr{D}(\mathscr{A})$ such
that $M$ is ghost projective. Then $\gamma$ is zero in
$\mathscr{D}(\mathscr{A})$.
\end{lem}

\begin{proof}
We can write $\gamma =\Theta(g)\Theta(f)^{-1}$ for some morphisms of
DG $\mathscr{A}$-modules fitting into a diagram of the form
\begin{equation*}
\xymatrix{
  & W \ar[dl]_f \ar[dr]^g& \\
   M&    &N    }
 \end{equation*}
where $f$ is a quasi-isomorphism. Since $\Theta(f)$ is an
isomorphism in $\mathscr{D}(\mathscr{A})$ and $\gamma$ is a ghost, we can
conclude that $H(g)=0$. By the assumption that $M$ is ghost
projective, there is a semi-projective resolution $\varepsilon:
F_M\stackrel{\simeq}{\to} M$  such that $F_M$ is a direct summand of
a DG free $\mathscr{A}$-module. Since $f$ is a quasi-isomorphism and $F_M$ is
homotopically projective, there is a morphism $h: F_M\to W$ such that
$f\circ h \sim \varepsilon$. It is easy to check $h$ is also a
quasi-isomorphism. Therefore $H(g\circ h)=0$. By Lemma
\ref{summand}, $g\circ h$ is zero in $\mathrm{K}(\mathscr{A})$. Obviously,
$\Theta(g\circ h)= \Theta(g)\circ \Theta(h)=0$. Since $\Theta(h)$ is
an isomorphism in $\mathscr{D}(\mathscr{A})$, we may conclude that
$\Theta(g)=0$ in $\mathscr{D}(\mathscr{A})$. Therefore $\gamma=0$ in
$\mathscr{D}(\mathscr{A})$.
\end{proof}

By Lemma \ref{ghostproj}, each ghost projective DG $\mathscr{A}$-module has
ghost length $0$. Hence, one might think that the ghost length of a
DG $\mathscr{A}$-module in DG context is analogous to the projective dimension
of a module in homological ring theory.

\begin{prop}\label{ghlen}
Let $M$ be a compact left DG module over a connected cochain DG algebra $\mathscr{A}$.  Then
$\mathrm{gh.len}_{\mathscr{A}}M=n$ if and only if $M\in \langle \mathscr{A}\rangle_{n+1}\setminus \langle\mathscr{A}\rangle_{n}$.
\end{prop}

\begin{proof}
To prove the assertion above, it suffices to prove that $M\in
\langle\mathscr{A}\rangle_{n+1}$ if and only if $\mathrm{gh.len}_{\mathscr{A}}M\le
n$.

We first prove that $\mathrm{gh.len}_{\mathscr{A}}M\le n$ if $M\in
\langle \mathscr{A}\rangle_{n+1}$. If $M\in \langle
\mathscr{A}\rangle_1$, then $M$  is isomorphic in $\mathscr{D}(\mathscr{A})$
to a direct summand of a certain DG free $\mathscr{A}$-module with a finite DG free basis. By Lemma \ref{ghostproj}, for any $N\in \mathscr{D}(\mathscr{A})$ and ghost morphism
$M\to N$ is zero. Hence $\mathrm{gh.len}_{\mathscr{A}}M\le 0$. The case that
$n=0$ is done. Suppose inductively that we have proved the cases
that $n\le k-1, k\ge 1$. We need to prove the case that $n=k$.
Suppose that $M\in \langle \mathscr{A}\rangle_{k+1}$. Then $M$ is a
direct summand of a DG $\mathscr{A}$-module $F$, which lies in $\langle
\mathscr{A}\rangle_{k}\star \langle \mathscr{A}\rangle_1$. There is
a triangle  $$ F_0\stackrel{\alpha}{\to} F \stackrel{\beta}{\to}
F_1\stackrel{\gamma}{\to}\Sigma F_0 \eqno(1)$$  in $\mathscr{D}(\mathscr{A})$,
where $F_0,F_1$ are in $\langle \mathscr{A}\rangle_k$ and $\langle
\mathscr{A}\rangle_1$ respectively. In order to prove that $\mathrm{gh.dim}_{\mathscr{A}}
M \le k$, it suffices to show that the composite of any $k+1$ ghost
morphisms
$$M\stackrel{f_0}{\to} N_1\stackrel{f_1}{\to} N_2\stackrel{f_2}{\to}\cdots \stackrel{f_{k-1}}{\to}
N_{k}\stackrel{f_{k}}{\to} N_{k+1}$$ is zero in $\mathscr{D}(\mathscr{A})$. Let
$F=M\oplus M'$ in $\mathscr{D}(\mathscr{A})$. We define a morphism $g_0: F\to
N_1$ by $g_0|_M=f_0$ and $g|_{M'}=0$. Clearly, $g_0$ is also a ghost
morphism. Hence $g_0\circ \alpha$ is also a ghost morphism since $H(g_0\circ \alpha)=H(g_0)\circ H(\alpha)=0$. The composite of $k$ ghost morphisms
$$F_0\stackrel{\alpha}{\to}F\stackrel{g_0}{\to}N_1 \stackrel{f_1}{\to}
N_2\stackrel{f_2}{\to}\cdots \stackrel{f_{k-2}}{\to}N_{k-1}
\stackrel{f_{k-1}}{\to} N_{k}$$ is zero by the induction hypothesis.
Since the sequence  $$ \Hom_{\mathscr{D}(\mathscr{A})}(F_1,N_{k})\to
\Hom_{\mathscr{D}(\mathscr{A})}(F,N_{k})\to \Hom_{\mathscr{D}(\mathscr{A})}(F_0,N_{k})$$
is exact at the second term $\Hom_{\mathscr{D}(\mathscr{A})}(F,N_{k})$, there
is a morphism $\theta: F_1\to N_{k}$ in $\mathscr{D}(\mathscr{A})$ such that
$$f_{k-1}\circ f_{k-2}\circ \cdots\circ  f_1\circ g_0= \theta \circ
\beta .$$ The composite  $f_{k}\circ \theta$ is a ghost morphism by
the assumption that $f_{k}$ is a ghost morphism. Since $F_1\in
\langle \mathscr{A}\rangle_1$, the ghost morphism $f_{k}\circ \theta$ is zero in
$\mathscr{D}(\mathscr{A})$ by the induction hypothesis. Therefore, the
composite of $k$ ghost morphisms $$f_{k}\circ f_{k-1}\circ
\cdots\circ  f_1\circ g_0= f_{k}\circ \theta \circ \beta $$ is zero
in $\mathscr{D}(\mathscr{A})$.  By the definition of $g_0$, we conclude that
the composite of $k+1$ ghost morphisms $f_{k}\circ f_{k-1}\circ
\cdots\circ  f_1\circ f_0$ is a zero morphism in $\mathscr{D}(\mathscr{A})$.
Thus $\mathrm{gh.len}_{\mathscr{A}}M\le k$.

Now, it remains to show $M\in \langle \mathscr{A}\rangle_n$ when
$\mathrm{gh.len}_{\mathscr{A}}M\le n-1$.
 If $n=1$, then any ghost morphism from $M$
is a zero morphism in $\mathscr{D}(\mathscr{A})$ since $\mathrm{gh.len}_{\mathscr{A}}M=0$.
Since $M$ is compact, $H(M)$ is a finitely generated $H(\mathscr{A})$-module.
By Lemma \ref{homologysurj}, there is a DG free $\mathscr{A}$-module $F$  and a
DG morphism $f: F\to M$ such that $F$ has a finite DG free basis and $H(f)$ is surjective.
There is an
exact triangle
$$F\stackrel{f}{\to} M\stackrel{g}{\to} I \stackrel{h}{\to} \Sigma
F $$ in $\mathscr{D}(\mathscr{A})$. The morphism $g$ is $0$ in $\mathscr{D}(\mathscr{A})$ since it
is in fact a ghost morphism by the fact that $H(f)$ is surjective.
For $M$, we have the following trivial exact triangle
$$M\stackrel{\mathrm{id}}{\to} M \to 0 \to \Sigma M.$$
By [TR3], there is a map $\eta: M\to F$ such that the following
diagram
\begin{align*}
\xymatrix{M \ar[d]^{\mathrm{id}}\ar[r] & 0\ar[d]^0\ar[r] & \Sigma M\ar@{-->}[d]^{\Sigma(\eta)}\ar[r]^{\mathrm{id}} &\Sigma M \ar[d]^{\mathrm{id}}  \\
M\ar[r]^{g} & I\ar[r]^{h} &\Sigma F\ar[r]^{\Sigma f}& \Sigma M \\
}
 \end{align*}
commutes.  This implies that $M$ is a retract (direct summand) of
$F$. Hence $M\in \langle \mathscr{A}\rangle_1$. The case $n=1$ is
done.

 Suppose inductively that we have proved the case $n\le k, k\ge
1$. We need to prove the case $n=k+1$. By Lemma \ref{homologysurj},
there is a DG morphism $f_0: F_0\to M$ such that $H(f_0)$ is
surjective and $F_0$ is a DG free $\mathscr{A}$-module with a finite DG free basis. This induces an exact triangle $$
F_0\stackrel{f_0}{\to}M\stackrel{g_0}{\to} I_1\to \Sigma F_0$$ in
$\mathscr{D}(\mathscr{A})$, where $g_0$ is a ghost morphism and $I_1\in \mathscr{D}^c(\mathscr{A})$.  In the same way,
we can construct inductively, a sequence of exact triangles
\begin{align*}
F_1\stackrel{f_1}{\to}I_1 &\stackrel{g_1}{\to} I_2\to \Sigma F_1 \\
F_2\stackrel{f_2}{\to}I_2 & \stackrel{g_2}{\to} I_3\to \Sigma F_2 \\
\cdots\quad\quad &\cdots \quad\quad\cdots \quad \\
\quad \quad \quad \quad F_{k-1}\stackrel{f_{k-1}}{\to} I_{k-1} &
\stackrel{g_{k-1}}{\to}
I_k\to \Sigma F_{k-1} \\
\quad \quad \quad \quad F_k \stackrel{f_{k}}{\to}I_k &
\stackrel{g_{k}}{\to}
I_{k+1}\to \Sigma F_{k}, \\
\end{align*}
where each $I_i\in \mathscr{D}^c(\mathscr{A})$,  $F_i$ is a DG free $\mathscr{A}$-module with a finite DG free basis and $g_i$ is a ghost morphism,
for any $i=1, 2,\cdots, k$. Let $$\rho_i=g_i\circ g_{i-1}\circ \cdots
g_1\circ g_0,\, i =1, 2,\cdots, k.$$ We have a sequence of exact
triangles $$M\stackrel{\rho_i}{\to} I_{i+1} \to T_{i+1} \to \Sigma M
, \quad i=1,2,\cdots, k.
$$
By the octahedral axiom,  there is an exact triangle $$\Sigma F_0
\to T_2 \to \Sigma F_1 \to \Sigma^2 F_0$$ fitting into the following
commutative diagram
\begin{align*} \xymatrix{
M\ar[d]^{\mathrm{id}}\ar[r]^{g_0} &I_1\ar[r]\ar[d]^{g_1} & \Sigma F_0 \ar[r]\ar@{-->}[d] &\Sigma M\ar[d]^{\Sigma(\mathrm{id})} \\
M\ar[r]^{\rho_1}& I_2\ar[r]\ar[d] & T_2 \ar[r]\ar@{-->}[d] &\Sigma M \\
&\Sigma F_1\ar@{=}[r]\ar[d] &\Sigma F_1\ar@{-->}[d] &\\
&\Sigma I_1 \ar[r] &\Sigma^2 F_0 & \\
}.
 \end{align*}
This implies that $T_2\in \langle \mathscr{A} \rangle_2$. We assume
inductively that $T_i \in \langle
\mathscr{A}\rangle_i$ for any $2\le i\le l$. There exists an exact triangle
$$T_l\to T_{l+1}\to \Sigma F_l \to \Sigma T_1$$ fitting into the
following commutative diagram
\begin{align*} \xymatrix{
M\ar[d]^{\mathrm{id}}\ar[r]^{\rho_{l-1}} &I_l\ar[r]\ar[d]^{g_l} & T_l \ar[r]\ar@{-->}[d] &\Sigma M\ar[d]^{\Sigma(\mathrm{id})} \\
M\ar[r]^{\rho_{l}}& I_{l+1}\ar[r]\ar[d] & T_{l+1} \ar[r]\ar@{-->}[d] &\Sigma M \\
&\Sigma F_l\ar@{=}[r]\ar[d] &\Sigma F_l\ar@{-->}[d] &\\
&\Sigma I_l \ar[r] &\Sigma T_l & \\
}
 \end{align*}
by using the octahedral axiom again.  Hence $T_{l+1}\in \langle
\mathscr{A}\rangle_{l+1}$ by the induction hypothesis. Then we
prove $T_{k+1}\in \langle \mathscr{A}\rangle_{k+1} $ by the
induction above.

On the other hand, since $\mathrm{gh.len}_{\mathscr{A}} M \le k$, the composite
of $k+1$ ghost morphisms $\rho_{k}=g_k\circ g_{k-1}\circ \cdots
g_1\circ g_0$ is zero in $\mathscr{D}(\mathscr{A})$. By [TR3], there is a map
$\eta: \Sigma M\to T_{k+1}$ such that the following diagram
\begin{align*}
\xymatrix{M \ar[d]^{\mathrm{id}}\ar[r] & 0\ar[d]^0\ar[r] & \Sigma M\ar@{-->}[d]^{\eta}\ar[r]^{\mathrm{id}} &\Sigma M \ar[d]^{\mathrm{id}}  \\
M\ar[r]^{\rho_{k}} & I_{k+1}\ar[r] &T_{k+1}\ar[r]^{\Sigma f}& \Sigma M \\
}
 \end{align*}
commutes. Hence $\Sigma M$ is a retract of $T_{k+1}$ in
$\mathscr{D}(\mathscr{A})$. So $M\in \langle \mathscr{A}\rangle_{k+1} $ and we
prove the case $n=k+1$. Thus $M\in \langle \mathscr{A}\rangle_{n}$
if $\mathrm{gh.len}_{\mathscr{A}}M\le n-1$.

\end{proof}

\begin{prop}\label{clghlen}
 Assume that $\mathscr{A}$ is a connected cochain DG algebra and $M$ is an object in $\mathscr{D}^c(\mathscr{A})$. Then $\langle \mathscr{A}\rangle_n=\mathrm{add}(\mathscr{A})^{\star n}$ and $\mathrm{cl}_{\mathscr{A}}M=\mathrm{gh.len}_{\mathscr{A}}M$.
\end{prop}

\begin{proof}
Obviously, $\mathrm{add}(\mathscr{A})^{\star n}\subseteq  \langle\mathscr{A}\rangle_n$ since $\langle \mathscr{A}\rangle_n=\mathrm{smd}(\mathrm{add}(\mathscr{A})^{\star n})$. Conversely,
 any object $M$ in $\langle \mathscr{A}\rangle_n=\mathrm{smd}(\mathrm{add}(\mathscr{A})^{\star n})$ is isomorphic to a direct summand of DG $\mathscr{A}$-module $N \in \mathrm{add}(\mathscr{A})^{\star n}\subseteq \mathscr{D}^c(\mathscr{A})$.
By Proposition \ref{conelen}, we have $\mathrm{cl}_{\mathscr{A}}N\le n-1$. By Proposition \ref{clfree}, $N$ admits a minimal semi-free resolution $F_N$ with $\mathrm{DGfree. class}_{\mathscr{A}}F_n=t\le n-1$. The DG module $M$ is isomorphic to a DG module $P$ which is a direct summand of $F_N$.
Since $F_N$ is a finitely generated semi-free DG $\mathscr{A}$-module, $P$ is a semi-free DG module by Theorem \ref{semiqst}.

Let $F_N=P\oplus Q$, for some DG $\mathscr{A}$-module $Q$.
Since $\mathrm{DGfree.class}_{\mathscr{A}}F_N=t$,
$F_N$ admits a strictly increasing semi-free filtration:
$$0=F_N(-1)\subset F_N(0)\subset F_N(1)\subset  \cdots \subset F_N(t-1)\subset
F_N(t)=F_N, $$
where each quotient $F_N(i)/F_N(i-1)$ is a DG free $\mathscr{A}$-module.  Let $E_i=\{e_{i_j}|j\in I_i\}$ be a DG free basis of $F_i=F_N(i)/F_N(i-1), 0\le i\le t$. Then $\partial_{F_N}(e_{i_j})\subseteq \mathscr{A}(\bigsqcup\limits_{j=0}^{i-1}E_j)$,
 each graded free $\mathscr{A}^{\#}$-module $F_N(r)^{\#}$ can be decomposed as
 $$\bigoplus\limits_{i=0}^r\bigoplus\limits_{j\in I_i}\mathscr{A}^{\#}e_{i_j}, 0\le r\le t.$$
Let $e_{i_j}=p_{i_j}+q_{i_j}$, where $p_{i_j}\in P$ and $q_{i_j}\in Q$, for any $j\in I_i, i=0,1,\cdots,t$. We have
$$F_N(r)=[\sum\limits_{i=0}^r(\sum\limits_{j\in I_i}\mathscr{A}p_{i_j})]\oplus [\sum\limits_{i=0}^r(\sum\limits_{j\in I_i}\mathscr{A}q_{i_j})], r=0,1,\cdots, t.$$
Hence $$F_r=F_N(r)/F_N(r-1)=(\sum\limits_{j\in I_r}\mathscr{A}\overline{p_{i_j}})\oplus(\sum\limits_{j\in I_r}\mathscr{A}\overline{q_{r_j}}), r=0, 1,\cdots, t.$$
By Lemma \ref{useful}, each $\sum\limits_{j\in I_r}\mathscr{A}\overline{p_{i_j}}$ is either a zero module or a DG free $\mathscr{A}$-module.
Let $\omega_{r_{\lambda}},\lambda\in \Lambda_r$ be its DG free basis ($\Lambda_r=\emptyset$ and $\omega_{r_{\lambda}}=0$ if $\sum\limits_{j\in I_r}\mathscr{A}\overline{p_{r_j}}=0$). Then $$\sum\limits_{j\in I_r}\mathscr{A}\overline{p_{r_j}}=\bigoplus\limits_{j\in \Lambda_r}\mathscr{A}\omega_{r_j}.$$ Note that
$$\bigoplus\limits_{i\in I_r}\mathscr{A}^{\#}e_{r_i}=\bigoplus\limits_{\lambda\in \Lambda_r}\mathscr{A}^{\#}\omega_{r_{\lambda}}\oplus (\sum\limits_{j\in I_r}\mathscr{A}\overline{q_{r_j}})^{\#}$$ is a graded $\mathscr{A}^{\#}$-submodule of $F^{\#}$. So $\bigoplus\limits_{\lambda\in \Lambda_r}\mathscr{A}^{\#}\omega_{r_{\lambda}}$ is also a graded $\mathscr{A}^{\#}$-submodule of $F_N^{\#}$. Since $\partial_{F_r}(\omega_{r_{\lambda}})=0$, we have
$$\partial_{F_N}(\omega_{r_{\lambda}})\in F_N(r-1)\cap \sum\limits_{i=0}^r(\sum\limits_{j\in I_i}\mathscr{A}p_{i_j})=\sum\limits_{i=0}^{r-1}(\sum\limits_{j\in I_i}\mathscr{A}p_{i_j}),   r=0,1,\cdots, t.$$
Let $P(r)=\sum\limits_{i=0}^r(\sum\limits_{j\in I_i}\mathscr{A}p_{i_j}), r=0,1,\cdots,t$. Then
$$0\subseteq P(0)\subseteq P(1)\subseteq P(2)\subseteq \cdots \subseteq P(t-1)\subseteq P(t)=P$$ is a filtration of DG $\mathscr{A}$-submodules of $P$. Moreover, $P(r)/P(r-1)=\sum\limits_{j\in I_r}\mathscr{A}\overline{p_{r_j}}$ is either zero or a DG free $\mathscr{A}$-module $\bigoplus\limits_{j\in \Lambda_r}\mathscr{A}\omega_{r_j}$. If $P(r)/P(r-1)=0$, for some $r\in \mathbb{N}$, then we just cancel such $P(r)$. In this way, we can get a strictly increasing semi-free filtration of $P$. Then $\mathrm{DGfree.class}_{\mathscr{A}}P\le t$ and hence $M\in \mathrm{add}(\mathscr{A})^{\star t+1}\subseteq \mathrm{add}(\mathscr{A})^{\star n}$.

\end{proof}

 \begin{defn}
For an DG $\mathscr{A}$-module $M$, its $\mathscr{A}$-level is defined to be
$$\mathrm{level}_{\mathscr{D}(\mathscr{A})}^{\mathscr{A}}(M)=\inf\{n\in \Bbb{N}\cup \{0\}| M\in
\langle \mathscr{A}\rangle_n\}.$$
\end{defn}
 This invariant is
originally introduced by Avramov, Buchweitz, Iyengar and Miller in
\cite{ABIM}. The $\mathscr{A}$-level of $M$ counts the number of steps required to build
$M$ out of $\mathscr{A}$ via triangles in $\mathscr{D}(\mathscr{A})$. It is a very
important numerical invariant in the study of the derived category
of compact DG $\mathscr{A}$-modules. By Proposition \ref{ghlen}, we can easily get the following proposition.
\begin{prop}\label{ghlenlevel}
For any $M\in \mathscr{D}^c(\mathscr{A})$, we have $\mathrm{gh.len}_{\mathscr{A}}M=\mathrm{level}_{\mathscr{D}(\mathscr{A})}^{\mathscr{A}}(M)-1$.
\end{prop}

In \cite{HL1,HL2}, Hovey-Lockridge introduced and studied ghost dimension for rings. In a similar way,  we define the left ghost dimension of $\mathscr{A}$ as
$$l.\mathrm{gh.dim}(\mathscr{A}) = \sup\{\mathrm{gh.len}_{\mathscr{A}}M | M\in
\mathscr{D}^c(\mathscr{A})\}.$$
Similarly, we can define the right global dimension of $\mathcal{A}$ as $$r.\mathrm{gh.dim.}\mathscr{A} =
\sup\{\mathrm{gh.len}_{\mathscr{A}^{op}}M | M\in \mathscr{D}^c(\mathscr{A}^{op})\}.$$

\begin{rem}
Applying the results above, we can get the characterizations of $l.\mathrm{gh.dim}(\mathscr{A})$ in terms of the invariants cone length and level.
Let $\mathscr{A}$ be a connected cochain DG algebra. Then
\begin{align*}
l.\mathrm{gh.dim}(\mathscr{A}) & = \sup\{\mathrm{cl}_{\mathscr{A}}M | M\in
\mathscr{D}^c(\mathscr{A})\} \\
&=\sup\{\mathrm{level}_{\mathscr{D}(\mathscr{A})}^{\mathscr{A}}(M)-1| M\in \mathscr{D}^c(\mathscr{A})\}.
\end{align*}
\end{rem}

\subsection*{Acknowledgments}
The work is supported by the National Natural Science Foundation of
China (No. 11871326), by Key Disciplines of Shanghai Municipality(No. S30104), and the Innovation Program of
Shanghai Municipal Education Commission (No. 12YZ031).

\def\refname{References}

\end{document}